\documentclass[final,leqno]{siamltex}

\usepackage{amsmath, amssymb,color}
\usepackage{graphicx}

\title{Positive steady states of evolution equations with finite dimensional nonlinearities}

\author{\`{A}ngel Calsina\thanks{Department of Mathematics, Universitat Aut\`{o}noma de Barcelona,  Bellaterra, 08193,  Spain 
({\tt acalsina@mat.uab.cat}).}
        \and J\'{o}zsef Z. Farkas\thanks{Division of Computing Science and Mathematics,
     University of Stirling, Stirling, FK9 4LA, Scotland ({\tt jozsef.farkas@stir.ac.uk}).}}

\begin{document}

\maketitle

\begin{abstract}
We study the question of existence of positive steady states of nonlinear evolution equations. We recast the steady state equation in the form of eigenvalue problems for a parametrised family of unbounded linear operators, which are generators of strongly continuous semigroups; and a fixed point problem. In case of irreducible governing semigroups we consider evolution equations with non-monotone nonlinearities of dimension two, and we establish a new fixed point theorem for set-valued maps. In case of reducible governing semigroups we establish results for monotone nonlinearities of any finite dimension $n$. In addition, we establish a non-quasinilpotency result for a class of strictly positive operators, which are neither irreducible nor compact, in general. We illustrate our theoretical results with examples of partial differential equations arising in structured population dynamics. In particular, we establish existence of positive steady states of a size-structured juvenile-adult and a structured consumer-resource population model, as well as for a selection-mutation model with distributed recruitment process.
\end{abstract}



\pagestyle{myheadings}
\thispagestyle{plain}
\markboth{\`{A} CALSINA AND J. Z. FARKAS}{POSITIVE STEADY STATES OF EVOLUTION EQUATIONS}

\section{Introduction of the problem}
One of the fundamental questions in the qualitative theory of nonlinear evolution equations is the existence of non-trivial  steady states (time independent solutions). A point of case are partial differential equations arising in structured population dynamics. One naturally expects that under some biologically reasonable conditions a population dynamics model admits at least one (or often exactly one) non-trivial steady state. The same may apply to evolution equations describing physical processes. 
In this work our aim is to establish existence of positive steady states for a class of nonlinear evolution equations. We are motivated mainly by studying partial differential equations arising in structured population dynamics, however, the results we establish here are quite general and they are applicable to a wide range of problems.

In previous studies, researchers used a bifurcation theoretic framework to establish existence of local and global branches of positive equilibria of age-structured population models. In \cite{C85,C85-2} Cushing used a net reproduction number, which is basically a density dependent net reproduction function evaluated at the extinction steady state, as a bifurcation parameter to establish existence of positive steady states. This approach is biologically inspired. 
More recently, in \cite{W1,W2,W3} Walker used bifurcation theory to establish existence of local and global branches of positive steady states of nonlinear age-structured population dynamical models with diffusion. We do not follow their approach here, instead we employ some completely dif\-fe\-rent ideas. The method we develop here is based on the reformulation of the steady state problem as an abstract eigenvalue problem for a family of unbounded operators coupled with a fixed-point problem for an appropriately defined nonlinear map. This approach was recently applied for some models which we were able to treat in a relatively straightforward fashion, see \cite{CS,CF,CP,FGH}. 
In the present paper we develop a general framework which can be applied to a wide class of models. The development of the general framework  requires to establish some new fixed point and spectral theoretic results.  

Let $\mathcal{X},\mathcal{Y}$ be Banach lattices, and consider the abstract Cauchy problem:
\begin{equation}\label{problem}
\frac{d u}{d t}=\mathcal{A}_{\bf u}\,u,\quad u(0)=u_0,
\end{equation}
where for every ${\bf u}\in \mathcal{Y}$, $\mathcal{A}_{\bf u}$ is a linear operator with an appropriately defined domain
$D(\mathcal{A}_{\bf u})\subseteq\mathcal{X}$. The relationship between ${\bf u}\in\mathcal{Y}$-parameter space, and $u\in\mathcal{X}$-state space, is determined by the environmental operator: 
\begin{equation*}
E\,:\,\mathcal{X}\to\mathcal{Y},\quad E(u)={\bf u}.
\end{equation*} 
This latter condition we call the environmental condition, or feedback. If the range of $E$ is contained in $\mathbb{R}^n$ for some $n\in\mathbb{N}$ then we say that problem \eqref{problem} incorporates a (finite)
$n$-dimensional nonlinearity. In the concrete applications it is often the case that $E$ is a positive, linear, and bounded functional, for example an integral operator, such as in the case of the partial differential equations we will discuss in Section 4. The advantage of the formulation \eqref{problem} above is that the (positive) steady state problem can be cast in the simple form:
\begin{equation}\label{problemss}
\frac{d u}{d t}=0=\mathcal{A}_{\bf u}\, u,\quad E(u)={\bf u};\quad 0\ne u\in\mathcal{X}_+.
\end{equation}
Recently we discussed some problems with one-dimensional nonlinearities, i.e. when the range of $E$ is contained in $\mathbb{R}$, see \cite{CS,FGH}. In general, if $\mathcal{A}_{\bf u}$ is a generator of a strongly continuous semigroup of bounded linear operators for every ${\bf u}\in\mathcal{Y}$, then it is enough to guarantee that the operator $\mathcal{A}_{\bf u}$ has eigenvalue zero with  at least one corresponding positive eigenvector. Note that in case of one-dimensional nonlinearities the environmental condition
$E(u)={\bf u}\in\mathbb{R}$ (at the steady state) is automatically satisfied (at least for a linear operator $E$), since the eigenvector(s) corresponding to the zero eigenvalue is only determined up to a constant. 

For a large class of nonlinear  evolution equations, for example those arising in structured population dynamics, the semigroup generated by $\mathcal{A}_{\bf u}$ often has desirable properties. In particular, it is positive, and even irreducible under some natural conditions; and it is often eventually compact. There are a number of results characterising spectral properties of irreducible and compact operators, see e.g. \cite{AGG,Sch}.
In case of multidimensional non\-li\-ne\-arities, problem \eqref{problemss} can be treated, at least in some special cases, by combining the analysis of spectral properties of the linear operator $\mathcal{A}_{\bf u}$ with a  fixed point argument.
This is the case when the spectral bound function, ${\bf u}\to s(\mathcal{A}_{\bf u})$ is monotone along positive rays in $\mathcal{Y}$, and the semigroup generated by $\mathcal{A}_{\bf u}$ is irreducible. This approach was employed recently in
\cite{CF,CP} for models with infinite dimensional nonlinearities. Unfortunately, in case of several real world applications, the spectral bound function  is not monotone along positive rays. In this case we can still utilise spectral properties of the operator $\mathcal{A}_{\bf u}$ but we need to find a fixed point of a multivalued map. Fixed point results for multivalued maps available in the literature require, amongst other things, convexity (essentially) of the range of the maps (see e.g. \cite{GD}), which does not hold in our case, in general. It is our aim in this paper to establish some fixed point results for multivalued maps, which are new as far as we know; and to develop a general framework to treat the steady state problem \eqref{problemss} for models with finite dimensional nonlinearities. 

Further difficulties arise, when the semigroup generated by $\mathcal{A}_{\bf u}$ is not irreducible. In this case, first of all one has to prove that the spectrum of $\mathcal{A}_{\bf u}$ contains an eigenvalue (i.e. the semigroup is not quasi-nilpotent). We will establish at the end of Section 2 a quite general result, which guarantees non-quasinilpotency of the semigroup generated by $\mathcal{A}_{\bf u}$, i.e. that its spectrum is not empty.  On the other hand, the geometric multiplicity of the spectral bound may be greater than one in general, and we need to assure the existence of a fixed point of a set-valued map. In Section 2 we will establish some results which provide treatment to this problem at least for some models with finite dimensional nonlinearities. In Section 3 we reformulate the steady state problem via a fixed point problem defined on the state space (rather than on the parameter space). We briefly describe the advantages of this formulation.
In Section 4 we apply the general results developed in Section 2 to some partial differential equations. 
In particular, we consider some well-known models of physiologically structured populations, such as a size-structured juvenile-adult, and a structured consumer-resource model. We also establish existence of a positive steady state of a selection-mutation model with distributed recruitment process.

\section{General results}

In this section we establish some abstract results which provide a general
framework for treating the steady state problem
\eqref{problemss} for models with finite di\-men\-sional nonlinearities.
First we will consider the case when $\mathcal{A}_{\bf
u}$ generates a positive irreducible semigroup for every ${\bf
u}\in\mathbb{R}^2_+$, but the model ingredients are such that they give rise to 
nonlinearities such that the spectral bound function ${\bf u}\to s(\mathcal{A}_{\bf u})$ is not monotone along positive rays. 

\subsection{$\mathcal{A}_{\bf u}$ generates an irreducible semigroup}

Let $\Sigma=\{(x,y)\in\mathbb{R}^2_+\,|\,
x+y=1\}$, $S$ be any subset of $\mathbb{R}^2_+$ and
$\pi\,:\,S\to\Sigma$  be the projection along positive rays. Note
that $\pi^{-1}\,:\, \Sigma\multimap S$ is multivalued, in general.
\begin{lemma}\label{fixedpointirred}
Let $S\subset\mathbb{R}^2_+$ be connected, such that
$S\cap\{(0,y)\,|\, y\in (0,R)\}\ne\emptyset$, and $S\cap\{(x,0)\,|\,
x\in (0,R)\}\ne\emptyset$ for some $R > 0$. Let $G\,:\, S\to\Sigma$
be continuous. Then the multivalued map
$H=\pi^{-1}\circ G\,:S\multimap S$ has at least one fixed
point, i.e. there exists $s\in S$ such that $s\in
\pi^{-1}(G(s)).$
\end{lemma}
\begin{proof} We define the set
\begin{equation}
C=\{(x,y)\in\Sigma\times\Sigma\,|\, x\in\Sigma,\,\, y\in G(\pi^{-1}(x))\}.
\end{equation}
First notice that $C$ is a connected set. Indeed,
$C=h(S)$ where $h\,:\,S\to\Sigma\times\Sigma$ defined as
$h(s)=(\pi(s),G(s))$ for $s\in S$, is a continuous function, since
both of its components are continuous. Next we note that there
exists a $y_1\in (0,R)$ such that $(0,y_1)\in S$
and $h(0,y_1)=(\pi((0,y_1)),G((0,y_1)))=((0,1),G((0,y_1)))$
therefore
$h(S)\cap(\{(0,1)\}\times\Sigma)\ne\emptyset$.
Similarly, there exists an $x_1\in (0,R)$ such that
$(x_1,0)\in S$ and
$h(x_1,0)=(\pi((x_1,0)),G((x_1,0)))=((1,0),G((x_1,0)))$
hence
$h(S)\cap(\{(1,0)\}\times\Sigma)\ne\emptyset$. It
follows that there exists a $\lambda\in [0,1]$ such
that $((\lambda,1-\lambda),(\lambda,1-\lambda))\in h(S)$ since
otherwise, $h(S)$ decomposes in the union of two disjoint non-empty
relatively open sets, namely,
$A:=\{((\lambda,1-\lambda),(\mu,1-\mu))\in h(S):\lambda
>\mu\}$ and $h(S) \setminus A$. Hence there exists an $s\in S$ such that
$\pi(s)=(\lambda,1-\lambda)=G(s)$, that is $s\in \pi^{-1}(G(s))$, and the proof is completed.
\qquad\end{proof}

{\sc Remark}
Note that even in the case of a multivalued function $G\,:S\multimap \Sigma$, the existence of a fixed point of the multivalued map $\pi^{-1}\circ G$ is equivalent to the existence 
of a fixed point of the multivalued map $G\circ\pi^{-1}$, where the composed maps are indicated in the following diagram: 
\begin{equation}
S\overset{G}{\multimap}\Sigma\overset{\pi^{-1}}{\multimap}S\overset{G}{\multimap}\Sigma.
\end{equation}

\begin{lemma}\label{levelset}
Let $f\,:\,\mathbb{R}^2_+\to\mathbb{R}$ be a continuous function, and assume that $f(0,0)>0$, and that there exists an $R>0$ such that $f(x,y)<0$ for $x^2+y^2>R^2$. Then the set
\begin{equation}
Z=\left\{{\bf z}\in\mathbb{R}^2_+\,|\, f({\bf z})=0\right\}
\end{equation}
has a compact connected subset $S$, which intersects the sets 
\begin{equation*}
\{(0,y)\,|\,r\le y\le R\},\,\, \text{and}\,\,\{(x,0)\,|\, r\le x\le R\},
\end{equation*} 
for some $r>0$.
\end{lemma}
\begin{proof}
We use a path crossing lemma (Lemma 2.9 in
\cite{PZ}) to show that the level set $Z$ has a compact connected subset $S$, which intersects the
sets: 
\begin{equation*}
\{(0,y)\,|\,r\le y\le R\},\,\, \text{and} \,\,\{(x,0)\,|\, r\le x\le R\},
\end{equation*} 
for some $r>0$ small enough. Let
$Q=[0,R+1]^2$ a rectangle, and let
\begin{align*}
F_l^-=& \{(x,0)\,|\, 0\le x\le r\}\cup\{(0,y)\,|\, 0\le y\le r\}, \\
F_r^-=& \{(x,0)\,|\, R\le x\le R+1\}\cup\{(0,y)\,|\, R\le y\le R+1\} \\
& \cup\{(x,R+1)\,|\, 0\le x\le R+1\}\cup\{(R+1,y)\,|\, 0\le y\le R+1\}, \\
F_b^+=& \{(x,0)\,|\, r\le x\le R\}, \\
F_t^+=& \{(0,y)\,|\,r\le y\le R\}.
\end{align*}
Then $\widetilde{Q}=(Q,Q^-)$ is an oriented rectangle according to
Definition 8 in \cite{PZ}, where $Q^-=F_l^-\cup F_r^-$, and clearly
$F_l^-$ and $F_r^-$ are disjoint. We note that $Z\subset Q$, and for
every path $\gamma\subset Q$ connecting $F_l^-$ and $F_r^-$ we have
$Z\cap\gamma\ne \emptyset$. It follows from Lemma 2.9 in \cite{PZ}
that there exists a compact connected set $S\subset
Z$ which intersects both $F_b^+$ and $F_t^+$. 
\qquad\end{proof}

In order to deal with continuous perturbations of
linear operators we will use the notion of generalised convergence, see \cite[Chapter 4.2]{Kato}. If $T$ and $S$ are closed linear operators from the Banach space $\mathcal{X}$ into itself,
then one can consider the graphs $G(T)$ and $G(S)$, which are closed linear manifolds
in the pro\-duct topology. The distance between two closed linear manifolds $M$ and $N$ is defined
as 
\begin{equation*}
\hat{d}(M,N)=\max\{d(M,N),d(N,M)\},
\end{equation*} 
where 
\begin{equation*}
d(M,N)=\displaystyle\sup_{u\in S_M}\displaystyle\inf_{v\in S_N}||u-v||.
\end{equation*} 
Above $S_N$ and $S_M$ denote the unit sphere of $N$ and $M$, respectively. 
We can define now the distance between $T$ and $S$ as 
\begin{equation*}
\bar{d}(T,S)=\hat{d}(G(T),G(S)).
\end{equation*} 
With this distance function the space $\mathcal{C(X)}$ of closed linear operators becomes a metric space, and convergence of a sequence $T_n$ is defined by $\bar{d}(T_n,T)\to 0$. We recall a criterion of generalised convergence for a sequence of operators. In particular, the following result is Theorem 2.25 from \cite[Chapter IV]{Kato}.
\begin{theorem}\label{genconv}
Let $T$ be a closed linear operator from the Banach space $\mathcal{X}$ into itself such that its resolvent set, denoted by 
$P(T)$, is not empty. The sequence of closed linear operators $T_n$ converges to $T$ in the generalised sense, if 
each $\lambda\in P(T)$ belongs to $P(T_n)$ for sufficiently large $n$, and: 
\begin{equation*}
||R(\lambda,T)-R(\lambda,T_n)||\to 0.
\end{equation*}
\end{theorem}
We will use this result later on in the Appendix. In the sequel of the proof of Theorem \ref{irreducible} below we will also use 
a fundamental result found in \cite[Chapter IV]{Kato}, concerning the continuity of a finite system of eigenvalues of 
closed linear operators. In particular, the continuity of the spectral bound function is a consequence of the following theorem, which is Theorem 3.16 in \cite[Chapter IV]{Kato}. 
\begin{theorem}\label{spectrumcont}
Let $T$ be a closed linear operator, and let the spectrum of $T$, denoted by $\sigma(T)$ separated into two parts $\sigma'(T),\, \sigma''(T)$ by a rectifiable, simple, closed curve $\Gamma$. Let $\mathcal{X}=M'(T)\oplus M''(T)$ be the associated spectral decomposition of $\mathcal{X}$. Then there exists an $\varepsilon>0$, depending on $T$ and $\Gamma$, with the following properties. Any closed linear operator $S$ with $\bar{d}(S,T)<\varepsilon$ has spectrum $\sigma(S)$ likewise separated by $\Gamma$ into two parts: $\sigma'(S)$ and $\sigma''(S)$. In the associated spectral decomposition $\mathcal{X}=M'(S)\oplus M''(S)$, $M'(S)$ and $M''(S)$ are isomorphic with $M'(T)$ and $M''(T)$, respectively. In particular, $\dim M'(S)=\dim M'(T)$, $\dim M''(S)=\dim M''(T)$ and both $\sigma'(S)$ and $\sigma''(S)$ are non-empty if this is true for $T$. The decomposition  $\mathcal{X}=M'(S)\oplus M''(S)$ is continuous in $S$ in the sense that the projection $P[S]$ of $\mathcal{X}$ onto $M'(S)$ along $M''(S)$ tends to $P[T]$ in norm as $\bar{d}(S,T)\to 0$.
\end{theorem} 

Our main result is the following.
\begin{theorem}\label{irreducible}
Assume that the linear operator $\mathcal{A}_{\bf
u}$ generates a positive irreducible and eventually compact
semigroup of bounded linear operators for every ${\bf
u}\in\mathbb{R}^2_+$, and that $\bf{u}_n \rightarrow
\bf{u}$ implies $\mathcal{A}_{\bf u_n} \rightarrow \mathcal{A}_{\bf
u}$ in the generalised sense. Moreover, assume that $s(\mathcal{A}_{\bf 0})>0$, and 
there exists an $R>0$ such that $s(\mathcal{A}_{\bf u})<0$ for $||{\bf u}||\ge R$. Further assume
that $E$ is a strictly positive linear functional.
Then the steady state problem \eqref{problemss} has a solution.
\end{theorem}
\begin{proof}
Theorem \ref{spectrumcont} above shows that generalised convergence implies that the spectral bound function $\sigma\,:\,\mathbb{R}^2_+\to\mathcal{C(X)}\to\mathbb{R}$ is continuous.
Let us define the level set
\begin{equation}\label{levelset2}
Z=\{{\bf z}\in\mathbb{R}^2_+\,|\, \sigma(z)=s(\mathcal{A}_{\bf z})=0\}.
\end{equation}
Lemma \ref{levelset} implies that 
the set $Z$ has a compact connected subset $S$, which intersects the sets
\begin{equation*}
\{(0,y)\,|\,r\le y\le R\},\,\, \text{and}\,\, \{(x,0)\,|\, r\le x\le R\},
\end{equation*} 
for some $r>0$.
We now apply Lemma \ref{fixedpointirred} with the
following setting: $G=\Pi\circ E\circ F\circ P$, where
\begin{equation}
G \,:\,\underbrace{{\bf v}}_{\in
S}\xrightarrow{P}\underbrace{\mathcal{A}_{\bf v}}_{\in
\mathcal{C}(\mathcal{X})}
\overset{F}{\to}\underbrace{u}_{\in \mathcal{X}\setminus \{0\}}
\xrightarrow{E}\underbrace{{\bf u}}_{\in \mathbb{R}^2_+\setminus
\{0\}} \overset{\Pi}{\to}\underbrace{{\bf v}'}_{\in \Sigma}.
\end{equation}

Here $P$ denotes the mapping which sends the
parameter value to the cor\-res\-pon\-ding (closed) generator, $F$ stands
for the map which gives the (uniquely determined) normalised
positive eigenvector corresponding to the spectral bound of the
generator (the spectral bound of the generator of a positive
irreducible and eventually compact semigroup is an algebraically
simple isolated eigenvalue with a corresponding strictly positive
eigenvector, see e.g. \cite[Theorem 9.10 and Theorem 8.17]{CH}), and
$\Pi$ is the projection along rays into the set $\Sigma$.

Note that $G$ is continuous since $P$ is continuous
by hypothesis and $F$ is also continuous (see Theorem 3.16 of Chap.
IV in \cite{Kato}). By Lemma \ref{fixedpointirred} we have a fixed point $s^*$ of the map
$H=\pi^{-1}\circ G$, which yields a unique positive steady
state as follows. There exists an $s^*\in S$
such that $G(s^*)=\left(\Pi\circ E\circ F\circ P\right)(s^*)=\pi(s^*)$. The
steady state is then given by $\lambda(F\circ
P)(s^*)$ for some $\lambda>0$ which has to satisfy
$E(\lambda(F\circ P)(s^*))=\lambda \left(E\circ F\circ P\right)(s^*)=s^*$. But $\left(\Pi\circ E\circ F\circ P\right)(s^*)=\pi(s^*)$ implies $\left(E\circ F\circ P\right)(s^*)\sim s^*$, where $x\sim y$ means that $x=cy$ for some $c>0$. 
\qquad\end{proof}

\subsection{$\mathcal{A}_{\bf u}$ generates a non-irreducible semigroup}

Next we discuss the more delicate case when the semigroup generated by $\mathcal{A}_{\bf u}$ is not irreducible,  
but still eventually compact. First we assume that the spectral bound function is monotone along positive rays. 
This assumption allows us to deal with equations with any finite dimensional nonlinearities, in general. 
Let $S$ be the level set where the spectral bound equals zero, and assume that $S$ intersects every ray in the 
positive cone of $\mathbb{R}^n$. Note that later we will impose different assumptions on the spectral bound 
function which will replace this latter hypothesis. Let us denote the $n$-dimensional 
simplex by 
\begin{equation*}
\Sigma=\{(x_1,\cdots,x_n)^T\in\mathbb{R}^n_+\,\vert\, x_1+\cdots +x_n=1\}.
\end{equation*} 
Notice that the hypothesis on the level set $S$ and the monotonicity of the spectral bound function imply that the projection along rays into $\Sigma$, denoted by $h$, establishes a homeomorphism between $S$ and $\Sigma$.

We define the nonlinear multivalued map $\Phi\,:\, S \multimap S$ as follows:
\begin{equation}\label{Phimap}
\Phi({\bf w})=\left(H\circ E\circ F\circ P\right)({\bf w})\subset S.
\end{equation}
The construction is described briefly on the following diagram:
\begin{equation}\label{Phimap2}
\Phi\,:\,\underbrace{{\bf w}}_{\in
S}\xrightarrow{P}\underbrace{\mathcal{A}_{\bf w}}_{\in
\mathcal{C}(\mathcal{X})} \overset{F}{\multimap}\underbrace{ N_{\bf w}}_{\subset \mathcal{X}_+} \xrightarrow{E}\underbrace{\{E(u),\,u\in
N_{\bf w}\}}_{\subset \mathbb{R}^n_+}
\xrightarrow{H}\underbrace{\displaystyle\mathop{\bigcup}_{u\in N_{\bf w}}H(E(u))}_{\subset S}.
\end{equation}
In \eqref{Phimap2}, as before, $P$ denotes the mapping which sends the parameter ${\bf w}$ 
to the parametrised closed generator $\mathcal{A}_{\bf w}$. We introduced above the notation 
$N_{\bf w}=\mathcal{S}_+\cap\text{span}\{u_i,\,i\in I_{\bf w}\}$, where $\mathcal{S}_+$ is the unit sphere intersected with the positive cone; and the $u_i$'s are the (linearly independent) eigenvectors corresponding to the 
spectral bound $s(\mathcal{A}_{\bf w})$. Hence the set-valued map $F$ assigns to each operator $\mathcal{A}_{\bf w}$ the linear span of the normalised non-negative eigenvectors corresponding to its spectral bound $0$. $H$ stands for the projection along rays into the set $S$. Note that since we assumed that the spectral bound is monotone along rays, 
$H$ is single valued. Also note that even when the semigroup generated by $\mathcal{A}_{\bf w}$ is not irreducible, but positive, its spectral bound still has at least one positive eigenvector (unless the semigroup is quasi-nilpotent). 
This is a consequence of the Krein-Rutman theorem applied to the positive semigroup (at some time $t_0$), and 
noting that the eigenvectors of the semigroup and its generator are the same.
If the algebraic multiplicity of the spectral bound $s\left(\mathcal{A}_{\bf w}\right)$ is greater than one, 
there may be multiple non-negative eigenvectors corresponding to the spectral bound.
In fact in this case there is a continuum of non-negative normalised eigenvectors 
corresponding to the spectral bound. If $u_1,u_2$ are non-negative normalised eigenvectors of $\mathcal{A}_{\bf w}$ 
for some ${\bf w}\in S$, then for any $\alpha\in [0,1]$ also $\alpha u_1+(1-\alpha)u_2$ is a non-negative eigenvector with norm $1$ if $\mathcal{X}$ is an AL-space, e.g. $L^1$. Hence for any $\mathcal{A}_{\bf w}$ the range of $F$ is non-empty, convex and closed. $E$ is well-defined and continuous and $P$ is continuous by hypothesis.  Recall that a correspondence $\Gamma\,:\,A\to B$ is called upper hemicontinuous (u.h.c. for short) at $a\in A$ if and only if for any open neighbourhood $V$ of $\Gamma(a)$ there exists a neighbourhood $U$ of $a$ such that for all $x$ in $U$, $\Gamma(x)$ is a subset of $V$.  
Our goal is to show that $F$ is upper hemicontinuous and apply Kakutani's fixed point theorem to the map $h\circ\Phi\circ h^{-1}$. Here we recall Kakutani's fixed point theorem in a general form for the readers convenience (see e.g. Corollary 10.3.10 in \cite{GD}).
\begin{theorem}\label{Kakutani}
Let $V$ be a locally convex topological space, and $C$ be a non-empty convex subset of $V$. Let $f\,: V\to V$ be an upper hemicontinuous  set valued map, such that $f(v)$ is closed and convex for all $v\in V$. Then $f$ has a fixed point.
\end{theorem}

Next we note that for any ${\bf w}\in\mathbb{R}^2_+$ if $\mathcal{A}_{\bf w}$ is a generator of an eventually compact semigroup then the algebraic multiplicity of the spectral bound (which is the dimension of the generalised eigenspace)  is finite, say $n$. It follows from Theorem \ref{spectrumcont} that in fact the algebraic multiplicity of the spectral bound is constant on the level set $S$, since the only continuous step function is the constant function.
\begin{lemma}\label{hemicont}
If for every ${\bf w}\in S$ the geometric multiplicity of the spectral bound of $\mathcal{A}_{\bf w}$ equals the algebraic multiplicity and the corresponding eigenspace has a basis of non-negative eigenvectors, then the set-valued map $F$ is upper hemicontinuous.
\end{lemma}
\begin{proof}
Let us recall the sequential characterisation of upper hemicontinuity of a correspondence $\Gamma:\,A\to B$. That is $\Gamma$ is u.h.c. at $a\in A$, if for any sequences $a_n\in A,\,b_n\in\Gamma(a_n)$, and $b\in B$, if 
$\displaystyle\lim_{n\to \infty}a_n=a$ and $\displaystyle\lim_{n\to\infty} b_n=b$ then $b\in \Gamma(a)$. 
Let ${\bf w}\in S$. We apply \cite{Kato}[Theorem 3.16, Sect.IV.] to show that $F$ is u.h.c at $\mathcal{A}_{\bf w}$. 
Let $\Sigma'(\mathcal{A}_{\bf w})=\left\{s(\mathcal{A}_{\bf w})\right\}$ and $\Sigma''(\mathcal{A}_{\bf w})$ consists of the rest of the spectrum. We have the decomposition $\mathcal{X}=M'(\mathcal{A}_{\bf w})\oplus M''(\mathcal{A}_{\bf w})$, where $M'(\mathcal{A}_{\bf w})$ and $M''(\mathcal{A}_{\bf w})$ are the generalised eigenspaces corresponding to 
the spectral sets $\Sigma'(\mathcal{A}_{\bf w})$ and $\Sigma''(\mathcal{A}_{\bf w})$, respectively. 
Note that our assumptions imply that in fact $M'(\mathcal{A}_{\bf w})\cap\mathcal{S}_+=N_{\bf w}=F(\mathcal{A}_{\bf w})$. 
Now let $\mathcal{A}_{\bf w_n}$ be a sequence converging to $\mathcal{A}_{\bf w}$ in the generalised sense, and 
${\bf u}_{\bf w_n}\in F(\mathcal{A}_{\bf w_n})$ a sequence of eigenvectors converging to ${\bf u_w}$. Similarly, let 
$\Sigma'(\mathcal{A}_{\bf w_n})=\left\{s(\mathcal{A}_{\bf w_n})\right\}$ and $\Sigma''(\mathcal{A}_{\bf w_n})$ consists of the rest of the spectrum, and $\mathcal{X}=M'(\mathcal{A}_{\bf w_n})\oplus M''(\mathcal{A}_{\bf w_n})$. 
\cite{Kato}[Theorem 3.16, Sect.IV.] states that the decomposition $\mathcal{X}=M'(\mathcal{A}_{\bf w_n})\oplus M''(\mathcal{A}_{\bf w_n})$ is continuous in $\mathcal{A}_{\bf w_n}$ in the sense that the spectral projection $P(\mathcal{A}_{\bf w_n})$ onto $M'(\mathcal{A}_{\bf w_n})$ along $M''(\mathcal{A}_{\bf w_n})$ 
converges to $P(\mathcal{A}_{\bf w})$ in the generalised sense. This implies that 
$\text{dim}(M'(\mathcal{A}_{\bf w_n}))=\text{dim}(M'(\mathcal{A}_{\bf w}))$. Hence we have necessarily that ${\bf u_w}\in M'(\mathcal{A}_{\bf w})\cap\mathcal{S}_+=N_{\bf w}=F(\mathcal{A}_{\bf w})$, i.e. ${\bf u}_{\bf w}$ is a normalised non-negative eigenvector of $\mathcal{A}_{\bf w}$.
\qquad\end{proof}

\begin{proposition}\label{l1}
Let $S$ be as above and $\Phi$ is defined in \eqref{Phimap}-\eqref{Phimap2}. 
Then every fixed point $s^*$ of $\Phi$ determines at least one positive
solution of the steady state problem \eqref{problemss}.
\end{proposition}
\begin{proof} Note that the assumption that $S$ intersects every positive ray 
implies that the map $\Phi$ is well defined. Let $s^*\in S$ such that $s^*\in \Phi(s^*)$,
i.e. $s^*\in H\circ E\circ F\circ
P(s^*)$. It follows that there exists $x^*\in E\circ F\circ P(s^*)$
such that $s^*=H(x^*)$, which implies
that there exists $u^*\in F\circ P(s^*)$ such that $E(u^*)=x^*$.
As before, we only need to show that there exist a $\lambda>0$ such that $E(\lambda
u^*)=s^*$, i.e., such that $E(\lambda u^*)=\lambda E(u^*)=\lambda x^*=s^*$, which holds for some 
unique positive $\lambda$ since
$s^*=H(x^*)$ and $H$ is the projection into $S$ along positive rays. Clearly $\lambda u^*$ is a positive solution 
of \eqref{problemss}. \qquad\end{proof}

\begin{theorem}\label{nonirreducible}
Assume that the linear operator $\mathcal{A}_{\bf
u}$ generates a positive and eventually compact
semigroup of bounded linear operators for every ${\bf
u}\in\mathbb{R}^n_+$ and that $\bf{u}_n \rightarrow
\bf{u}$ implies $\mathcal{A}_{\bf u_n} \rightarrow \mathcal{A}_{\bf
u}$ in the generalised sense. Moreover, assume that $s(\mathcal{A}_{\bf 0})>0$, and 
there exists an $R>0$ such that $s(\mathcal{A}_{\bf u})<0$ for $||{\bf u}||\ge R$; and that the spectral 
bound function is monotone along rays in the positive cone. Further assume
that $E$ is a strictly positive linear functional and the assumptions of Lemma \ref{hemicont} hold true. 
Then the steady state problem \eqref{problemss} has a solution.
\end{theorem}
\begin{proof} If $F$ is u.h.c. then so is $h\circ\Phi\circ h^{-1}$. Next we note that $h\circ H=h$ and the image of a convex set by $E$ is convex. Furthermore, the projection of a convex subset 
of $\mathbb{R}^n_+$ into $\Sigma$ via $h$ is convex, too, hence we have that for any ${\bf x}\in\Sigma$ the set $(h\circ\Phi\circ h^{-1})({\bf x})$ is convex in $\Sigma$.  Theorem \ref{Kakutani} guarantees the existence of a fixed point ${\bf x}$ of the map $h\circ\Phi\circ h^{-1}$, which in turn implies the existence of a fixed point $h^{-1}({\bf x})$ of the map $\Phi$. Therefore, by Proposition \ref{l1} we obtain the existence of a positive steady state of \eqref{problem}.
\end{proof}

{\sc Remark}
Note that we may define the map $F$ such that it maps the operator $\mathcal{A}_{\bf w}$ 
into the generalised eigenspace intersected with the unit positive sphere. Even when the geometric multiplicity of the spectral bound is less than the algebraic multiplicity, $F$ is still shown to be u.h.c by Theorem \ref{spectrumcont}. But we may obtain as a fixed point a vector which is not a steady state. Also to establish u.h.c. of $F$ we would still need to guarantee that the generalised eigenspace $M(s(\mathcal{A}_{\bf w}))$ is spanned by non-negative vectors.

\medskip

The second somewhat independent characterisation of the existence of the positive steady state is based on the following Lemma. This characterisation works in the case of a non-monotone spectral bound function, but only for two-dimensional nonlinearities, and the main assumption seems to be very difficult to verify in the case of concrete models.

\begin{lemma}\label{fixgen}
Let $\gamma\subset\mathbb{R}^2_+$ and assume that there exists a homeomorphism $h$ such that $h(\gamma)=[0,1]$. Let $\Phi\,:\,\gamma\multimap \gamma$ be a set valued map. If the set
\begin{equation}\label{Kdef}
K=\{(x,y)\in[0,1]^2\,|\, y\in h(\Phi(h^{-1}(x)))\}
\end{equation}
is connected, then $\Phi$ has a fixed point, i.e.,
there exists a $q^*\in \gamma$ such that $q^*\in \Phi(q^*)$.
\end{lemma}
\begin{proof} First let us note that $K\cap\{(0,y)\,|\,y\in
[0,1]\}\ne\emptyset$ since it contains the point $(0,y)$ whenever $y\in h(\Phi(h^{-1}(0)))\ne \emptyset$. 
Similarly,
$K\cap\{(1,y)\,|\, y\in [0,1]\}\ne\emptyset$ because it contains $(1,y)$ whenever $y\in h(\Phi(h^{-1}(1)))\ne \emptyset$. 
Since $K$ is connected, it intersects the diagonal of $[0,1]\times [0,1]$, i.e. there exists a $x^*\in [0,1]$ such that $x^*\in h(\Phi(h^{-1}(x^*))).$ Equivalently $h^{-1}(x^*)\in\Phi(h^{-1}(x^*))$, which means that
$q^*=h^{-1}(x^*)$ is a fixed point of $\Phi$ on $\gamma$.
\qquad\end{proof}

Proposition \ref{l1} and Lemma \ref{fixgen} yield the following.
\begin{corollary}\label{l2} Assume that $s(\mathcal{A}_{\bf 0})>0$, and 
there exists an $R>0$ such that $s(\mathcal{A}_{\bf u})<0$ for $||{\bf u}||\ge R$. 
If the set $S$, given by Lemma \ref{levelset} applied to the spectral bound function, is homeomorphic to the interval $[0,1]$ and the set $K$ defined in 
\eqref{Kdef} (with $\gamma=S$) is connected, then the steady state problem \eqref{problemss} has a solution.
\end{corollary}

\subsection{Intermezzo - A non-quasinilpotency result}

We note that, in the case when the semigroup generated by $\mathcal{A}_u$ is irreducible and eventually compact it follows  that $\mathcal{A}_u$ has an eigenvalue, see e.g. \cite[C-III Theorem 3.7]{AGG}. However, this is not necessarily the case when $\mathcal{A}_u$ does not generate an irreducible semigroup or the semigroup is not eventually compact. Note that our main assumptions in Theorems \ref{irreducible} and \ref{nonirreducible} are that the spectral bound $s(\mathcal{A}_{\bf u})$ changes sign along positive rays, which requires the existence of spectral values of the operators $\mathcal{A}_{\bf u}$.  Hence, to apply the machinery developed in the previous subsections, one has to guarantee in the first place that the spectrum of the generator contains at least one eigenvalue. Here we present a quite general result, which we believe is interesting enough on its own right, since it shows non-quasinilpotency for a wide class of operators without assuming irreducibility or compactness. 

In the previous subsections we mainly dealt with generators of strongly continuous semigroups. Here we will not restrict ourselves to semigroups or generators, hence to start with we introduce some basic notions and recall some definitions. 
Let $\mathcal{X}$ be a normed vector lattice and $\mathcal{L}$ be a (non-zero) bounded linear operator on $\mathcal{X}$.
The closed non-empty set $K\subset\mathcal{X}$ is called a cone if it satisfies:
\begin{enumerate}
\item $\forall f,g\in K$ and $\alpha,\beta\ge 0$: $\alpha f+\beta g\in K$,
\item $K\cap -K=\left\{0\right\}$,
\item $K\ne \{0\}$.
\end{enumerate}
We introduce the notations $K_0:=K\setminus\{0\}$ and $\mathcal{X}_0:=\mathcal{X}\setminus\{0\}$.
$\mathcal{L}$ is called $K$-positive if $\mathcal{L}(K)\subseteq K$, and we say that $\mathcal{L}$ is strictly $K$-positive if $\mathcal{L}(K_0)\subseteq K_0$. We say that $\mathcal{L}$ is positive (resp. strictly positive) if $K=\mathcal{X}_+$, the positive cone of $\mathcal{X}$.
A positive operator $\mathcal{L}$ is called (ideal) irreducible if it does not admit invariant ideals other than the trivial ones: $\{0\}$ and $\mathcal{X}$.
Note that irreducibility implies strict positivity (except the case of the zero operator in one dimension), but not vice versa, in fact strict positivity is a much
weaker condition than irreducibility, in general. The ideal $J$ of the lattice $\mathcal{X}$ is called a band if it is closed under supremum, and the operator $\mathcal{L}$ is called band irreducible if the only invariant bands are the trivial ones.
As usual, let $r(\mathcal{L})=\sup\{|\lambda|\,|\,\lambda\in\sigma(\mathcal{L})\}$ denote the spectral radius.
$\mathcal{L}$ is called quasi-nilpotent (or topologically nilpotent) if $r(\mathcal{L})=0$. $\mathcal{X}$ is called an AL-space (abstract Lebesgue space) if for every $f,g\in \mathcal{X}_+$ one has $||f+g||=||f||+||g||$, and $\mathcal{X}$ is called an AM-space if for every $f,g\in\mathcal{X}_+$ one has $||f\vee g||=||f||\vee ||g||$.  For further terminology not introduced here we refer to \cite{Sch}.

Spectral properties of positive irreducible operators have been studied extensively.  In particular, results about the non-quasinilpotency of irreducible operators on AL- and AM-spaces can be found already in \cite{Sch2}. Ben de Pagter proved in \cite{Pagter} that a compact irreducible operator on a Banach lattice (of dimension greater than one) is not quasi-nilpotent, i.e. its spectral radius is strictly positive. This result was extended in \cite{Sch3} for band irreducible operators. Note that a result which asserts that the spectrum of the generator of a strongly continuous semigroup on a Banach lattice is not empty,  also requires compactness and irreducibility, see Theorem 3.7 in \cite[Sect. C-III]{AGG}; or at least band irreducibility, see \cite{Sch4}. Band irreducibility is a somewhat weaker condition than irreducibility, in general.

Our goal here is to establish a quite general result, which does not require irreducibility or compactness of the operator. We only have to assume strict positivity and a condition about the compactness of the image of a specific set.
This condition however, seems to be independent of irreducibility, and does not imply compactness of the operator. Our result is obtained using a simple argument combined with Schauder's fixed point theorem, which we recall here for the reader's convenience in its most general form.

\begin{theorem}\label{Schauder} (\cite{Cauty})
Let $V$ be a separated (i.e. Hausdorff) topological vector space. Let $C$ be a non-empty convex subset of $V$ and
$f$ a continuous function from $C$ to $C$. If $f(C)$ is contained in a compact subset of $C$ then $f$ has a fixed point.
\end{theorem}
The following is our general result.
\begin{theorem}\label{theorem1}
Let $\mathcal{X}$ be a normed vector lattice and $K$ be a cone. Let
$\mathcal{S}$ denote the unit sphere of $\mathcal{X}$ and let $C=Conv(\mathcal{S}\cap
K)$, the convex hull of $\mathcal{S}\cap K$. Let $\mathcal{L}$ be a strictly $K$-positive bounded linear operator.
If $\mathcal{L}(C)$ is relatively compact in $\mathcal{X}_0$, then
$\mathcal{L}$ has a positive eigenvalue with a corresponding eigenvector from $K$.
\end{theorem}
\begin{proof}
First note that $C$ is bounded since it is a convex hull of a bounded set. Also note that $C\subset K_0$.
We define a map $\Phi\,:\, C\to C$ as $\Phi=\mathcal{P}\circ\mathcal{L}$, where $\mathcal{P}(k)=\frac{k}{||k||}$, for $k\in K_0$. $\mathcal{L}$ is continuous on $\mathcal{X}$ and $\mathcal{P}$ is continuous on $\mathcal{X}_0$ therefore $\Phi$ is continuous on $\mathcal{X}_0$. Since $\Phi(C)$ is relatively compact in $\mathcal{X}_0$, $\overline{\Phi(C)}$ is compact, moreover $\Phi(C)\subseteq\overline{\Phi(C)}\subseteq (\mathcal{S}\cap K)\subseteq C$. We apply Theorem \ref{Schauder} to the map $\Phi$. A fixed point $x\in C$ of this map $\Phi$ in turn implies the existence of a fixed-ray $R=\left\{\alpha\,x\,\vert\,x\in (\mathcal{S}\cap K),\,\alpha\ge 0\right\}$ of the  operator $\mathcal{L}$. On this ray $R$ the operator $\mathcal{L}$ acts as a multiplication operator with a positive constant $||\mathcal{L}(x)||$, i.e. $\mathcal{L} x=||\mathcal{L}(x)|| x$. Hence $||\mathcal{L}(x)||$ is a positive eigenvalue of $\mathcal{L}$ with a corresponding positive eigenvector which is any (non-zero) element of the fixed-ray $R$ (hence it is contained in $K$).
 \qquad\end{proof}

Note that strict positivity is necessary even in finite dimension. For instance, consider the matrix $L=\begin{pmatrix}
0 & 0 \\ 1 & 0 \end{pmatrix} $, with $K=\mathbb{R}^2_+$. Then $L$ is positive but not strictly positive, and it is nilpotent.

We also note that the Krein-Rutman theorem asserts that if $K-K$ is dense in $\mathcal{X}$ and 
$\mathcal{L}$ is a compact and $K$-positive operator such that $r(\mathcal{L})>0$ then the spectral radius has a corresponding eigenvector from $K$.
Note that Theorem \ref{theorem1} implies that the spectral radius is positive, and that $r(\mathcal{L})$ is a positive spectral value, but not necessarily an eigenvalue. On the other hand the main corollary of Theorem \ref{theorem1} is that $r(\mathcal{L})>0$ and it also applies to cases where $K-K$ is not dense in $\mathcal{X}$.

\section{Fixed point problem formulated on the state space}
Let us recall here the general formulation \eqref{problemss} of the steady state problem. 
\begin{equation*}
\frac{d u}{d t}=0=\mathcal{A}_{\bf u}\, u,\quad E(u)={\bf u}\in\mathcal{Y}_+;\quad 0\ne u\in\mathcal{X}_+.
\end{equation*}
To summarise the developments in the previous sections we note that our general strategy is to decouple the steady state problem into a linear problem, which essentially amounts to assuring the existence of a positive eigenvector corresponding to the zero eigenvalue of the linear operator $\mathcal{A}_{\bf u}$; and a nonlinear problem, that is guaranteeing the existence of a fixed point of a nonlinear map. Clearly, most of the difficulties arise in the second part of the problem. Note that the first problem was formulated on the state space $\mathcal{X}$, while the  second, the  nonlinear problem, was formulated on the parameter space $\mathcal{Y}$. There are good reasons why it is natural to define a fixed point map in the parameter space. First, it is relatively straightforward to impose conditions on the model ingredients, which guarantee the existence of a "nice" zero level set, i.e. where the spectral bound of the generator $\mathcal{A}_{\bf u}$ equals zero. Also the dimension of the range of the environmental operator $E$ determines the dimension of the nonlinearity. 

It is worth discussing though, that equivalently, we may formulate a fixed point problem on the state space $\mathcal{X}$. This approach, as we will see, works well for cases when the semigroup generated by $\mathcal{A}_{\bf u}$ is not irreducible but the spectral bound function is monotone along positive rays. To illustrate this approach first in a relatively simple case assume that for every ${\bf u}$ the semigroup generated by $\mathcal{A}_{\bf u}$ is irreducible, and that the spectral bound function ${\bf u}\to s(\mathcal{A}_{\bf u})$ is monotone along positive rays in $\mathcal{Y}$. This is for example the case of the semilinear partial differential equation we discussed in \cite{CF}. Since we would like to show the existence of a positive steady state we further assume that $s(\mathcal{A}_{\bf 0})>0$, and that there exists an $R>0$ such that $s(\mathcal{A}_{\bf u})<0$ for $||{\bf u}||\ge R$. 
Note that this condition is exactly the same as we imposed both in Theorem \ref{irreducible} and Theorem \ref{nonirreducible}. Further assume that $E$ is a positive, bounded linear operator, and that $\mathcal{X}$ is an AL-space (for example the Lebesgue space $L^1(0,m)$, as in \cite{CF}). Let us denote by $\mathcal{S}$ the unit sphere of $\mathcal{X}$, and let $\mathcal{S}_+=\mathcal{S}\cap\mathcal{X}_+$. It is shown that if $\mathcal{X}$ is an AL-space then $\mathcal{S}_+$ is convex. We can define a continuous nonlinear map $T$ from $\mathcal{S}_+$ into itself. To this end we define a map $\eta\,:\, \mathcal{S}_+\to \eta(\mathcal{S}_+)\subset \mathcal{X}_+$ as follows: for every $x\in \mathcal{S}_+$ let $\eta(x)=cx$, where $0<c$ is the unique real number such that $s(\mathcal{A}_{[c\cdot E(x)]})=s(\mathcal{A}_{[E(cx)]})=0$. The existence of such $c$ follows from the fact that the spectral bound changes sign along every positive ray in $\mathcal{Y}$, while the uniqueness of $c$ follows from the monotonicity of the spectral bound and the linearity of $E$. Note that $\eta$ is a projection along rays, in fact it is shown that $\eta$ is continuous. The definition of the nonlinear map $T$ is illustrated in the following diagram, where $P$ and $F$ are as before in Section 2, i.e. $F$ gives the normalised positive eigenvector corresponding to the spectral bound of the generator 
$\mathcal{A}_{[E(p)]}$.
\begin{equation}
T\,:\,\underbrace{x}_{\in \mathcal{S}_+}\xrightarrow{\eta}\underbrace{p}_{\in \mathcal{X}_+}
\overset{E}{\to}\underbrace{E(p)}_{\in \mathcal{Y}_+}
\xrightarrow{P}\underbrace{\mathcal{A}_{[E(p)]}}_{\in \mathcal{C(X)}}\overset{F}{\to}\underbrace{x'}_{\in \mathcal{S}_+}.
\end{equation}
$T(\mathcal{S}_+)$ is contained in a compact subset of $\mathcal{S}_+$ if for example the normalised eigenvectors of $\mathcal{A}_{[E(p)]}$ have appropriate regularity properties, e.g. because they belong to the domain of $\mathcal{A}_{[E(p)]}$. Then, Theorem \ref{Schauder} implies the existence of a fixed point of the map $T$. Note that if $x^*$ is a fixed point of $T$ then $\eta(x^*)$ is a positive solution of problem \eqref{problemss}.  

We are of course interested whether this approach works for example in the case when the semigroup generated by $\mathcal{A}_{\bf u}$ is not irreducible. It appears that in this case (but still assuming monotonicity of the spectral bound along positive rays) the only difference is that $F$ may be set valued, but the range of $F$ (defined as the generalised eigenspace intersected with the positive cone and the unit sphere) is non-empty, convex and closed for any $\mathcal{A}_{[E(p)]}$, where $p\in \eta(\mathcal{S}_+)$. 
Without elaborating the details we note that to establish existence of a fixed point of the map $T$ we may apply Himmelberg's theorem, which we recall below from \cite{H} for the readers convenience.
\begin{theorem}\label{him}
Let $U$ be a non-empty subset of a separated (i.e. Hausdorff) locally convex topological space $V$. Let $f\,:U\to U$ an upper hemicontinuous set-valued map, such that $f(v)$ is closed and convex for all $v\in V$, and $f(U)$ is contained in some compact subset $C$ of $U$. Then $f$ has a fixed point. 
\end{theorem}

To apply Theorem \ref{him} one needs to show that $T$ is upper hemicontinuous and that $T(\mathcal{S}_+)$ is contained in a compact subset of $\mathcal{S}_+$. The second condition may be verified for concrete applications using regularity properties of the eigenvectors of the generator. The upper hemicontinuity of $T$ may be verified (similarly as in Lemma \ref{hemicont}) at least in the special case when for every $x\in \mathcal{S}_+$ the geometric multiplicity of $s(\mathcal{A}_{[E(\eta(x))]})$ equals its algebraic multiplicity and the corresponding eigenspace has a basis of non-negative eigenvectors.

\section{Applications}

In this section we consider some model examples to illustrate the
abstract results in the previous section. The examples we consider here 
are nonlinear differential equation models of physiologically structured po\-pu\-lations. 

\subsection{Juvenile-adult model}
The first example we
discuss is a structured juvenile-adult model, recently investigated
for example in \cite{FH2}. Let $p(s,t)$ denote the density of
individuals of size (or another appropriate physiological structuring
variable) $s$ at time $t$. We assume that individuals enter the
adult stage at the fixed size $l$, and we assume that there is a
finite maximal size $m$. Hence the juvenile and adult population sizes at every
time $t$ are given by 
\begin{equation*}
J(t)=\int_0^l p(s,t)\,d s,\quad A(t)=\int_l^m p(s,t)\,d s,
\end{equation*}
respectively. We write a single
equation which governs the dynamics of the whole population with the
understanding that $\mu(s,\cdot,\cdot)$ and $\gamma(s,\cdot,\cdot)$ denote juvenile mortality and growth rate 
for $0\le s\le l$, while they denote adult mortality and growth rate for $l\le s\le
m$, respectively.  We also assume that only adults reproduce, hence the fertility function $\beta$ is supported 
over the size-interval $[l,m]$. With these model ingredients the governing equation reads
\begin{equation}\label{ja}
p_t(s,t)+(\gamma(s,J(t),A(t))\,p(s,t))_s=-\mu(s,J(t),A(t))\,p(s,t),
\end{equation}
which is defined for $0\le s \leq m<\infty$ and $t>0$, and it is subject to the
boundary condition
\begin{equation}\label{jaboundc}
\gamma(0,J(t),A(t))p(0,t)=\int_l^m\beta(s,J(t),A(t))\,p(s,t)\,ds, \quad t>0,
\end{equation}
and an initial condition of the form
\begin{equation}\label{jainc}
    p(s,0)=p_0(s), \quad 0\leq s\leq m.
\end{equation}
We define the operator
\begin{align}
\Psi_{(J,A)}\,p & =-(\gamma(\cdot,J,A)p(\cdot))_s-\mu(\cdot,J,A)p(\cdot), \label{jaop1} \\
D(\Psi_{(J,A)}) & =\{p\in W^{1,1}(0,m)\,|\,
p(0)=B_{(J,A)}(p)\}, \label{jaop2}
\end{align}
where
\begin{equation} \label{jaop3}
B_{(J,A)}(p)=\frac{1}{\gamma(0,J,A)}\int_l^m\beta(s,J,A)p(s)\,d
s, \quad D(B_{(J,A)})=L^1(l,m).
\end{equation}
It can be shown that $\Psi_{(J,A)}$ generates a positive and eventually compact semigroup for every $(J,A)\in\mathbb{R}^2_+$,
which is irreducible if there exists an $\varepsilon>0$ such that 
\begin{equation}\label{betairreducible}
\displaystyle\int_{m-\varepsilon}^m\beta(s,\cdot,\cdot)\,d s>0,
\end{equation} 
see e.g. \cite{FGH}. The (positive linear) environmental operator $E$ is defined as
\begin{equation*}
E(p)=\left(\int_0^lp(s)\,d s,\int_l^m p(s)\,d s\right)^t.
\end{equation*}

Let us assume that $\beta$, $\gamma$ and $\mu$ are
non-negative continuous functions of all of their arguments and that
$\gamma$ is bounded away from $0$, and let us consider the operator
$\Psi_E$, where $E=(J,A)$ as given by \eqref{jaop1}-\eqref{jaop3}.
For a given $E \in \mathbb{R}^2_+$, we shall prove in the Appendix
that $\Psi_{E_n}$ tends in the generalised sense to $\Psi_E$
whenever $E_n$ tends to $E$.

It follows from Theorem \ref{irreducible} that if $\beta$ satisfies \eqref{betairreducible}, and there exist po\-si\-ti\-ve real numbers $r,R$ such that $0<r<R<\infty$ and $s(\Psi_{(J,A)})>0$ for $J+A\le r$, and $s(\Psi_{(J,A)})<0$ for $J+A\ge R$ then model \eqref{ja}-\eqref{jainc} admits a positive steady state. Note that it is natural to assume that $\beta$ is strictly positive since the adult class consists of reproductive individuals by definition. The conditions on the spectral bound of the operator $\Psi_{(J,A)}$ are also natural.

{\sc Remark} \, We would like to note that the juvenile-adult model above can be treated in a different (perhaps more natural) way. It was shown in Proposition\,1 in \cite{FH2} that there is a one-to-one correspondence between positive steady states of the juvenile-adult model \eqref{ja}-\eqref{jainc} and pairs of positive numbers $(J_*,A_*)$ satisfying the following two scalar equations
\begin{equation}
{{J_*}\over {A_*}}= {\displaystyle{\int_0^l \exp\left\{-\int_0^s\displaystyle{\frac{\mu(r,J_*,A_*)
        +\gamma_s(r,J_*,A_*)}
        {\gamma(r,J_*,A_*)}}\,dr\right\}\,ds}\over
        \displaystyle{\int_l^m \exp\left\{-\int_0^s\displaystyle{\frac{\mu(r,J_*,A_*)+\gamma_s(r,J_*,A_*)}
        {\gamma(r,J_*,A_*)}}\,dr\right\}\,ds}},\quad  R(J_*,A_*)=1,\label{JAcond} 
\end{equation}    
where    
\begin{align}
& R(J,A)= \int_l^m\displaystyle{\frac{\beta(s,J,A)}{\gamma(s,J,A)}}
    \,\exp\left\{-\int_0^s\displaystyle{\frac{\mu(r,J,A)}{\gamma(r,J,A)}}\,dr\right\}\,ds.\label{Rdef}
\end{align}
Hence we may define a nonlinear multi-valued map, via the right hand-side of the first equation in \eqref{JAcond} (where the first and second component of the map are determined by the numerator and the denominator, respectively), on the positive quadrant, and apply Lemma \ref{fixedpointirred} and Lemma \ref{levelset}, with the level set $Z$ in Lemma \ref{levelset} defined via  $R(J_*,A_*)=1$, to this map. Then, essentially  the same conditions as in Theorem \ref{irreducible}, now formulated in terms of the density dependent net reproduction function $R$ defined in \eqref{Rdef}, will imply the existence of a positive steady state. In fact, the sufficient conditions on $R$ are the biologically relevant ones, i.e. that $R(0,0)>1$ and $R(J,A)<1$ for $J+A>\bar{R}$ for some $\bar{R}>0$. Note that these are the conditions equivalent to the ones on $f$ in Lemma \ref{levelset}.

\subsection{Consumer-resource model}

Next we turn our attention to a structured consumer-resource model. These types of models are frequently discussed in the literature, see e.g. \cite{dRP,DGMNR,FH1, FM, HT}. The reason behind this is that solutions are shown to exhibit periodic oscillations. This is because of the negative feedback due to consumption of the resource by the predator. We let $p(s,t)$ to denote the density of the consumer individuals of size $s$ at time $t$. Then $P(t)=\int_0^mp(s,t)\,d s$ is the consumer population size at time $t$, while $Q(t)$ denotes the total population size of the resource at time $t$. We consider the following model, see e.g. \cite{dRP,DGMNR,FH1}.
\begin{align}
& p_t(s,t)+(\gamma(s,P(t),Q(t))p(s,t))_s=-\mu(s,P(t),Q(t))p(s,t),  \label{cr1} \\
& \gamma(0,P(t),Q(t))p(0,t)=\int_0^m\beta(s,P(t),Q(t))p(s,t)\,d s,  \label{cr2} \\
& \frac{d Q(t)}{d t}=Q(t)f(Q(t))-\int_0^m F(s,P(t),Q(t))p(s,t)\,d s, \label{cr3}
\end{align}
with suitable initial conditions: $p(0,t)=p_0(s)$ and $Q(0)=Q_0$. We
assume a finite maximal size $m$ for the consumer population. The
fertility, mortality and growth rates, $\beta,\mu$ and $\gamma$ of
the consumers depend on their respective size and
on the consumer and resource population sizes. $f$ is a continuous
function which describes the dynamics of the resource if it is left
to its own devices, e.g. $f(Q)=(r-Q)$. Later on 
we will impose some natural assumptions on $f$ to guarantee the existence of a positive steady state. 
$F$ denotes the size-dependent feeding rate of the consumers. 

To apply the spectral theoretic framework one would naturally define a parametrised family of operators as
\begin{align}\label{consres-phidef}
& \Phi_{(P,Q)}{\bf v}= \begin{pmatrix}  -\left(\gamma(\cdot,P,Q)v(\cdot)\right)_s-\mu(\cdot,P,Q)v \\
Vf(Q)-\int_0^m F(s,P,Q)v(s)\,d s
\end{pmatrix},\quad  (P,Q)\in\mathbb{R}^2_+,
\end{align}
where ${\bf v}=(v,V)\in D(\Phi_{(P,Q)})=\{{\bf v}\in
W^{1,1}(0,m)\times\mathbb{R}\,|\,
v(0)=B_{(P,Q)}(v)\}$, and
\begin{equation}\label{defB}
B_{(P,Q)}(v)=\frac{1}{\gamma(0,P,Q)}\int_0^m\beta(s,P,Q)v(s)\,d
s,\quad D(B_{(P,Q)})=L^1(0,m).
\end{equation}
It turns out however, that $\Phi_{(P,Q)}$ does not generate a positive semigroup. 
In fact, we claim that defining the parametrised family of operators in a natural way would not yield generators of positive operators.

The existence of a positive steady state can be established, similarly as for the juvenile-adult model above, by directly applying Lemmas \ref{fixedpointirred}-\ref{levelset}. In particular, 
in \cite{FH1} we showed that the problem of the existence of
non-trivial steady states of \eqref{cr1}-\eqref{cr3} is equivalent
to the question of existence of non-trivial solutions of a system of
two non-linear equations, which involve some kind of net
reproduction rates. To formulate conditions to guarantee the
existence of such solutions was left as an open problem. More precisely, we proved in Proposition 1.1 in \cite{FH1} that there is a one-to-one correspondence between positive steady states of the consumer-resource model  \eqref{cr1}-\eqref{cr3} and pairs of positive numbers $(P_*,Q_*)$ satisfying the following two scalar equations
\begin{align}
\frac{P_*}{Q_*}= {f(Q_*)\displaystyle{\int_0^m \exp\left\{-\int_0^s\displaystyle{\frac{\gamma_s(r,P_*,Q_*)+\mu(r,P_*,Q_*)}
        {\gamma(r,P_*,Q_*)}}\,dr\right\}\,ds}\over
        \displaystyle{\int_0^m F(s,P_*,Q_*) \exp\left\{-\int_0^s\displaystyle{\frac{\gamma_s(r,P_*,Q_*)+\mu(r,P_*,Q_*)}
        {\gamma(r,P_*,Q_*)}}\,dr\right\}\,ds}},\quad  R(P_*,Q_*)=1,\label{PQcond} 
\end{align}    
where    
\begin{align}
& R(P,Q)= \int_0^m\displaystyle{\frac{\beta(s,P,Q)}{\gamma(s,P,Q)}}
    \,\exp\left\{-\int_0^s\displaystyle{\frac{\mu(r,P,Q)}{\gamma(r,P,Q)}}\,dr\right\}\,ds.\label{R2def}
\end{align}
Hence we may define a nonlinear multi-valued map, via the right hand-side of the first equation in \eqref{PQcond} (where the first and second component of the map are determined by the numerator and the denominator, respectively), on the positive quadrant, and apply Lemma \ref{fixedpointirred} and Lemma \ref{levelset}, with the level set $Z$ in Lemma \ref{levelset} defined via  $R(P_*,Q_*)=1$, to this map. Then, essentially  the same conditions as in Theorem \ref{irreducible}, now formulated in terms of the density dependent net reproduction function $R$ defined in \eqref{R2def}, will imply the existence of a positive steady state.

\subsection{Early human population model}
The next partial differential equation model we discuss arises from very recent developments to model the dynamics of the age-structured population of early humans, see \cite{BW,WG}. We consider a relatively simple equation, which still exhibits  
an essential feature, namely it is not governed by an irreducible semigroup. 
The governing equation reads:
\begin{equation}\label{human}
p_t(a,t)+p_a(a,t)=-(f(a)+\eta(a)T(t)+\mu(a)S(t))p(a,t),
\end{equation}
defined for $0<a<a_{max}<\infty$, and $t>0$. 
It is assumed that individuals have three different life stages: non-reproducing juvenile, reproducing adult, and non-reproducing senescent. Juveniles enter the reproducing adult stage at the fixed age $a_j$, and become infertile again upon reaching the fixed age $a_r$. Hence the senescent and total population sizes are given by:
\begin{equation*}
S(t)=\int_{a_r}^{a_{max}}p(a,t)\,d a,\quad T(t)=\int_0^{a_{max}}p(a,t)\,d a,
\end{equation*}
respectively. $f$ denotes the age-dependent natural mortality, $\eta(\cdot)T(t)$ denotes the extra mortality due to crowding/competition effects among all the individuals, while $\mu(\cdot)S(t)$ denotes the extra mortality due to the presence of senescent individuals (in [31] $\mu$ is assumed to vanish for  $a>a_j$ modelling an extra competitive pressure exerted by the senescent population on the non reproducing individuals). The extra mortality pressure on juveniles induced by the senescent population is due to limitation in resources, for example food. The influx of individuals at any time $t$ is determined by the fertility rate $\beta$, and the standing population, via the boundary condition:
\begin{equation}\label{humanboundc}
p(0,t)=\int_{a_j}^{a_r}\beta(a)\,p(a,t)\,d a, \quad t>0.
\end{equation}
We impose an initial condition of the form
\begin{equation}\label{humaninc}
    p(a,0)=p_0(a), \quad 0\leq a\leq a_{max}.
\end{equation}
We refer to \cite{WG} for a detailed analysis of the model (and in fact of a more general one). 
Our aim here is to illustrate our abstract results from the previous sections. 
We define the operator $\Psi$ as:
\begin{align}
\Psi_{(S,T)}\,p & =-p_a(\cdot)-(f(\cdot)+\eta(\cdot)T+\mu(\cdot)S)p(\cdot), \label{humanop1} \\
D(\Psi_{(S,T)}) & =\{p\in W^{1,1}(0,a_{max})\,|\, p(0)=B(p)\}, \label{humanop2}
\end{align}
where
\begin{equation*}
B(p)=\int_{a_j}^{a_r}\beta(a)p(a)\,d a, \quad D(B)=L^1(0,a_{max}).
\end{equation*}
It can be shown that for every $(S,T)\in\mathbb{R}^2_+$ the operator $\Psi_{(S,T)}$ generates a positive and eventually compact semigroup, which is not irreducible. The (positive, linear and bounded) environmental operator is defined as:
\begin{equation*}
E(p)=\left(\int_{a_r}^{a_{max}}p(a)\,d a,\int_0^{a_{max}}p(a)\,d a\right)^t.
\end{equation*} 
The spectral bound function $\sigma\, :\,(S,T)\to s\left(\Psi_{(S,T)}\right)$ is monotone decreasing in both 
variables, and it is natural to assume that $s\left(\Psi_{(0,0)}\right)>0$. 
It can be shown that $s\left(\Psi_{(S,T)}\right)<0$, for $S+T$ large enough. 

Despite the non-irreducibility of the semigroup generated by $\Psi_{(S,T)}$, it is shown by direct computation that the geometric and algebraic multiplicity of the spectral bound $s(\Psi_{(S^*,T^*)})=0$ equals $1$, for any $(S^*,T^*)\in Z$. We are under the hypotheses of Lemma \ref{hemicont}. Therefore Theorem \ref{nonirreducible} implies the existence of a positive steady state.

\subsection{Selection-mutation model}
The following general selection-mutation mo\-del for the dynamics of an age and age at maturity structured population was introduced and investigated very recently in \cite{CP}. 
\begin{align}
& \frac{\partial u}{\partial t}(l,a,t)+\frac{\partial u}{\partial a}(l,a,t)=-\mu(E[u],l,a)u(l,a,t),\nonumber \\
& u(l,0,t)=\int_0^\infty\int_{\hat{l}}^\infty b(l,\hat{l})\beta(E[u],\hat{l},a)u(\hat{l},a,t)\, da\,d\hat{l}, \label{selmut} \\
& u(l,a,0)=u_0(l,a). \nonumber
\end{align}
In the model above $u(l,a,t)$ denotes the density of individuals of age $a$ and maturation age $l$ at time $t$. $E$ is the environmental operator, as described in Section 1. $\mu$ denotes the mortality rate, $\beta$ the fertility function and 
$b$ the probability density function describing the mutation, that is $\int_{l_1}^{l_2} b(l,\hat{l})\, dl$ is the probability that the offspring of an individual of age $a$ at maturity $\hat{l}$ has maturity age $l\in (l_1,l_2)$. 
 
In \cite{CP} the existence and uniqueness of a positive steady state of model \eqref{selmut} was established, when the mortality $\mu$ is a strictly monotone increasing function of its first variable, while the fertility $\beta$ is a strictly monotone decreasing function of its first variable, as well. The result was established by formulating the positive steady state problem as an eigenvalue problem for a bounded integral operator, and studying spectral properties of that integral operator.

Here we apply our theory developed in the previous sections in a special case of model \eqref{selmut}, but without the crucial monotonicity assumptions on the mortality and fertility functions. We assume that individuals have a finite maximal age denoted by $a_m$ and that the range of the environmental operator is contained in $\mathbb{R}^2$. More precisely, we assume that 
\begin{equation*}
E[u(\cdot,\cdot,t)]=(P(t),Q(t))^t,
\end{equation*}
where
\begin{equation*}
P(t)=\int_0^{a_m}\int_0^lu(l,a,t)\, da\, dl,\quad Q(t)=\int_0^{a_m}\int_l^{a_m}u(l,a,t)\, da\,dl.
\end{equation*}
That is, we assume that the per capita vital rates depend on a two dimensional variable: the population of young individuals and that of adults. Hence again we have two-dimensional nonlinearities, but in contrast to the previous examples, a distributed recruitment process.

We define the operator $\Psi$ on the Banach space $\mathcal{X}=L^1((0,a_m),L^1(0,a_m))\cong L^1((0,a_m)\times(0,a_m))$ 
as follows
\begin{align}
\Psi_{(P,Q)}\, u & = -u_a(\cdot,\cdot)-m(P,Q,\cdot,\cdot)u(\cdot,\cdot), \label{selmutop1} \\
D(\Psi_{(P,Q)}) & =\{u\in L^1\left((0,a_m), W^{1,1}(0,a_m)\right)\,|\, u(\cdot,0)= B(u)\}, \label{selmutop2}
\end{align}
where
\begin{equation*}
B_{(P,Q)}(u)=\int_0^{a_m}\int_{\hat{l}}^{a_m}b(\cdot,\hat{l})\beta(P,Q,\hat{l},a)u(\hat{l},a)\, da\, d\hat{l},
\end{equation*}
is a bounded operator on $\mathcal{X}$ taking values in $L^1(0,a_m)$. 

To see that $\Psi_{(P,Q)}$ generates a positive semigroup for any $P,Q\ge 0$ we shall invoke the Lumer-Phillips Theorem (see e.g. \cite[Ch.II Th.3.15]{NAG}). To this end we need to show that for some large enough $\kappa>0$ the densely defined operator $\Psi_{(P,Q)}-\kappa\,\mathcal{I}$ is dissipative, and that the range condition holds true, that is rg$\left(\lambda\,\mathcal{I}-(\Psi_{(P,Q)}-\kappa\,\mathcal{I})\right)$ is dense in $L^1\left((0,a_m), L^1(0,a_m)\right)=\mathcal{X}$.
The dissipativity calculation is rather cumbersome and lengthy but relatively straightforward hence we do not include it here. 
Instead, we refer the reader for example to \cite{FGH}, where such a dissipativity calculation was carried out for a similar semigroup generator. To show that the range condition holds true we note for any $f\in \mathcal{X}$ the solution of the equation
\begin{equation}
\left(\lambda\,\mathcal{I}-\Psi_{(P,Q)}\right)u=f
\end{equation}
is 
\begin{equation}
u(\cdot,a)=e^{-\int_0^a(\mu(P,Q,\cdot,r)+\lambda)dr}\left(u(\cdot,0)+\int_0^a e^{\int_0^x(\mu(P,Q,\cdot,r)+\lambda)dr}f(\cdot,x)\, dx \right),\label{selmutressol}
\end{equation}
where $u(\cdot,0)$ satisfies the following equation
\begin{align}
u(\cdot,0)=& B_{(P,Q)}\left(u(\cdot,0)\,e^{-\int_0^a(\mu(P,Q,\cdot,r)+\lambda)dr}\right) \nonumber \\ 
+ & B_{(P,Q)}\left(\int_0^a e^{-\int_x^a(\mu(P,Q,\cdot,r)+\lambda)dr}f(\cdot,x)\, dx \right).\label{selmutressol2}
\end{align}
We define a bounded linear operator (for every $(P,Q)$, $\lambda>0$) mapping $L^1(0,a_m)$ into $\mathcal{X}$ by
\begin{equation*}
\left(S_{(P,Q)}\, v\right) (l,a)= e^{-\int_0^a(\mu(P,Q,l,r)+\lambda)dr}\, v(l).
\end{equation*}
This allows us to write equation \eqref{selmutressol2} as
\begin{equation}
\left(\mathcal{I}-B_{(P,Q)}S_{(P,Q)}\right) u(\cdot,0)=B_{(P,Q)}\left(\int_0^a e^{-\int_x^a(\mu(P,Q,\cdot,r)+\lambda)dr}f(\cdot,x)\, dx \right).\label{selmutressol3}
\end{equation}
It is shown that 
\begin{equation}
\left|\left|B_{(P,Q)}\,S_{(P,Q)}\right|\right|\le \left|\left|B_{(P,Q)}\right|\right|\,\left|\left| S_{(P,Q)}\right|\right|\le \frac{\displaystyle\sup \left\{\beta(P,Q,\cdot,\cdot)\right\}}{\lambda}.\label{selmutressol4}
\end{equation}
This shows that for $\lambda$ large enough the operator $\left(\mathcal{I}-B_{(P,Q)}\, S_{(P,Q)}\right)$ is invertible, hence for every $f\in\mathcal{X}_+$ there exists a non-negative solution of equation \eqref{selmutressol2}. Substituting this solution into \eqref{selmutressol} it is then shown by differentiating the solution $u$ in \eqref{selmutressol} with respect to $a$ that for $\lambda$ large enough we have $u\in D(\Psi_{(P,Q)})$, hence the range condition holds true. Since for any $P,Q\ge 0$, $B_{(P,Q)}$ is a positive operator it is clear that the semigroup generated by $\Psi_{(P,Q)}$ is positive.

Next we note that equation \eqref{selmutressol} shows that the solution $u(\cdot,\cdot)$ is strictly positive (i.e. $u>0$ almost everywhere) for $f\in\mathcal{X}_+$ if the recruitment operator $B_{(P,Q)}$ is strictly positive: i.e. $B_{(P,Q)}\,u>0$ almost everywhere, for $u\in \mathcal{X}_+$. This implies that the resolvent operator $R(\lambda,\Psi_{(P,Q)})$ is strictly positive, and the semigroup generated by $\Psi_{(P,Q)}$ is irreducible, see e.g. \cite{NAG}. Note that $\Psi_{(P,Q)}$ is a bounded perturbation of the generator of a translation semigroup. Since we have a finite maximal age $a_m$ this implies that the semigroup generated by $\Psi_{(P,Q)}$ is eventually compact. If $s(\Psi_{(0,0)})>0$ and there exists an $R>0$ such that $s(\Psi_{(P,Q)})<0$ for $P+Q\ge R$ then 
Theorem \ref{irreducible} implies the existence of a positive steady state. We note that the condition that the spectral bound function changes sign along positive rays in the parameter space, is very natural, and it corresponds to the conditions 
on the spectral radius of the integral operator analysed in \cite{CP}. In fact, there is a rigorous result, that is Theorem 3.5 in \cite{HT2}, which establishes the connection between the spectral bound of the operator $\Psi_{(P,Q)}$ and the spectral radius of the corresponding integral operator.

\section{Discussion}

In this paper we developed a general framework for proving existence of positive equilibria of nonlinear evolution equations. 
Our approach is based on a reformulation of the steady state equation as a parametrised family of abstract eigenvalue problems. 
Our spectral analysis relies heavily on perturbation results found in \cite{Kato}. Note that, although we asserted that the linear problems are go\-verned by strongly continuous semigroup of operators, which is the case for most of the concrete applications, from the mathematical point of view this assumption is not necessary. 
In principle, spectral properties of the operators $\mathcal{A}_{\bf u}$ in \eqref{problemss} may be 
studied without assuming that they generate semigroups. However, it proves to be convenient to work in the framework of the spectral theory of positive semigroups.  Also, as we have seen in Section 4 in some cases we may apply 
directly Lemmas \ref{fixedpointirred}-\ref{levelset} to prove existence of a positive steady state. The necessity of the 
spectral theoretic framework becomes apparent for models with infinite-dimensional nonlinearities, or for models (even with finite dimensional nonlinearities) when the existence of a positive steady state cannot be characterised explicitly via scalar equations involving the interaction variables. This is the case for example in the selection-mutation model discussed in Section 4.4, and in general in case of models with distributed states at birth. Such models were investigated recently for example in  \cite{AcklehFarkas, CF, CP, FGH}.

In Section 2.1 we discussed problems with non-monotone nonlinearities of dimension $2$.  The generalisation of these results 
for higher dimensions is an open problem and promises to be extremely difficult. This is mainly because in our fixed point theorem we made use of some geometric constraints only valid in the plane. Then in Section 2.2 we discussed problems with non-irreducible governing semigroups. In this case the spectral analysis is much more delicate. For example to establish upper hemicontinuity of the nonlinear fixed point maps we had to assume that the dimension of the generalised eigenspace of the spectral bound of the generator $\mathcal{A}_{\bf u}$ equals the dimension of its eigenspace. But we established results for equations with nonlinearities of any finite dimension $n$. Furthermore, as we briefly e\-la\-bo\-ra\-ted in Section 3 this approach can be extended to infinite dimensional nonlinearities using Himmelberg's fixed point theorem in \cite{H}, by formulating the fixed point problem on the state space. In Section 2.3 we also established a quite general spectral theoretic result, namely, non-quasinilpotency 
for a class of strictly positive operators. Note that, the fundamental result in \cite{Pagter} by de Pagter assumed both irreducibility and compactness of the operator. We impose neither of those assumptions, however we have to assume compactness 
of the image of a rather specific set.     

As we mentioned earlier, Cushing and Walker used a bifurcation theoretic approach to establish existence of positive steady states of structured population models, see e.g. \cite{C85,C85-2,W1,W2,W3}. The advantage of their approach is the clear biological motivation to use the inherent net reproduction rate (i.e. the density dependent net reproduction rate evaluated at the extinction steady state) as a bifurcation parameter. They, on the other hand, as far as we know only treated age-structured models. At the same time our main assumption in Theorems \ref{irreducible} and \ref{nonirreducible} that the spectral bound changes sign along positive rays is also very natural from the biological point of view. In fact we can clearly see the relationship between the two approaches. 
When the spectral bound of $\mathcal{A}_{\bf 0}$ becomes positive the trivial steady state looses its stability 
and a branch of non-trivial steady states bifurcates from it. On the other hand, naturally the spectral bound of $\mathcal{A}_{\bf u}$ will become negative far from the origin, due for example to limitations of resources. Hence the spectral bound function changes sign 
along positive rays, indeed.

Finally we note that for some models it is also possible to formulate directly a fixed point problem and to apply fixed point theorems in conical shells of Banach lattices, see e.g. \cite{FH,W2}. This approach however, requires the implicit solution of the steady state equation, which cannot be obtained for most models, for example for the ones with distributed recruitment processes, see e.g. \cite{AcklehFarkas, CF,CP,FGH}.

\section{Appendix}\label{app} 
Note that in Section 2 we recalled from \cite{Kato} and applied the notion of generalised convergence of a sequence of operators. In particular, we hypothesized in Theorems \ref{irreducible} and \ref{nonirreducible} 
that sequences of operators converge in the ge\-ne\-ra\-lised sense. Here we prove that this is actually the case 
for the operators arising in the juvenile-adult population model discussed in Section 4.1. 
The proof of generalised convergence for other concrete applications such as the ones we discussed in Sections 4.2-4.4 follows similar lines.

Let $\Psi_E$, where $E=(J,A)$, given by \eqref{jaop1}-\eqref{jaop3} and let $E_n
\rightarrow E$. We prove that $\Psi_{E_n}$ tends in the generalised
sense to  $\Psi_E$. To this end we utilise Theorem 2.3 (which is \cite[Theorem 2.25, Chap. IV]{Kato}), which characterizes convergence of operators in the generalised sense via convergence in norm of the resolvent operators. 
Indeed, let us assume that $\lambda$ is large
enough so that
\begin{equation}\label{resolv}
\int_l^m \frac{\beta(s,E)} {\gamma(s,E)}\exp\left\{-\int_0^s \frac{\lambda
+ \mu(z,E)}{\gamma(z,E)}\,d z\right\} \,d s < 1,
\end{equation}
and define, for a sequence $E_n$ with limit $E$ and any $f \in
L^1(0,m)$ with norm $1$, the following functions on $[0,m]:$
\begin{equation} 
\gamma_n(s)=\gamma(s,E_n),\quad \mu_n(s)=\mu(s,E_n),\quad h_n(s)=\frac{\beta(s,E_n)}{\gamma_n(s)},
\end{equation}
\begin{equation}   F_n(s)=e^{-\int_0^s\frac{\lambda + \mu_n(z)}{\gamma_n(z)} d z},\quad 
H_n(s)=\frac{F_n(s)}{\gamma_n(s)},
\end{equation}
and
\begin{equation}
G_n(s)=\int_0^s\frac{f(z)}{F_n(z)}\,d z,
\end{equation}
and the numerical sequence:
\begin{equation}
c_n=\frac{\int_l^m h_n(s)F_n(s)G_n(s)\,d s}{1 - \int_l^mh_n(s)F_n(s)\,d s}.
\end{equation}
Notice that $c_n$ is well defined for $\lambda$ large enough by
(\ref{resolv}). Similarly, but substituting $E_n$ by $E$, define
$\gamma(s),\,\mu(s),\,h(s),\,F(s),\,H(s)$,\,$G(s)$ and $c$. A
straightforward computation, based on the variation of constants
formula, shows that the resolvent operator of $\Psi_{E_n}$ is
explicitly given by: 
\begin{equation}
R_n(\lambda)f (s) = (c_n + G_n(s))H_n(s),
\end{equation}
with analogous expression for the resolvent of $\Psi_{E}.$ Alternatively, it is shown that $(c_n + G_n(\cdot))H_n(\cdot)\in D(\Psi_{E_n})$ and 
one can directly compute that
\begin{equation}
\partial_s(\gamma_n(s)H_n(s)) = \partial_s F_n(s) =-\frac{(\mu_n(s)+\lambda)}{\gamma_n(s)}F_n(s) = -
(\mu_n(s)+\lambda)H_n(s),
\end{equation}
and
\begin{equation}
\partial_s(\gamma_n(s)G_n(s)H_n(s)) = \partial_s (F_n(s)G_n(s)) =-(\mu_n(s)+\lambda)H_n(s)G_n(s) + f(s),
\end{equation}
which imply
\begin{align}
\partial_s (\gamma_n(s)(R_n(\lambda)f)(s)) &=\partial_s (c_n\gamma_n(s)H_n(s)+\gamma_n(s)G_n(s)H_n(s)) \nonumber \\
& =-(\mu_n(s)+\lambda)(c_n+G_n(s))H_n(s) + f(s) \nonumber \\
& =-(\mu_n(s)+\lambda)(R_n(\lambda)f)(s) + f(s),
\end{align}
i.e., $(\lambda - \Psi_{E_n})R_n(\lambda)f = f.$\\
Now it is easy to see that $c_n$ tends to $c$ and that $G_n
\rightarrow G$ in $L^1$, both uniformly with respect to $f$ of norm
1. For instance,
\begin{align}
 \|G_n - G\| & \leq \int_0^m \int_0^s |f(z)|\left|\frac{1}{F_n(z)} -
\frac{1}{F(z)}\right|\,d z\,d s \nonumber \\
& \leq m \int_0^m |f(z)|\left|\frac{1}{F_n(z)} -
\frac{1}{F(z)}\right|\,d z \nonumber \\ 
& \leq \sup_{z \in [0,m]}\left|\frac{1}{F_n(z)} -
\frac{1}{F(z)}\right| \rightarrow 0.
\end{align}
Finally, consider
\begin{align}
\|R_n(\lambda) - R(\lambda)\|\ & \leq \| (c_n + G_n)H_n - (c + G)H_n\| + \|(c + G)H_n - (c + G)H\| \nonumber \\
& \leq \sup_n \|H_n\|_\infty (|c_n-c| + \|G_n - G\|) + (c + \|G\|_\infty)\|H_n-H\| \nonumber \\
& \rightarrow 0
\end{align}
uniformly with respect to $f$ of norm 1, since $c$ and $G$ are both
bounded uniformly with respect to $f$ of norm 1. From this it
follows that $\Psi_{E_n}$ tends to $\Psi_E$ in the generalised sense, by Theorem \ref{genconv}.

\section*{Acknowledgements}
\`{A}. Calsina was partially supported by the research projects 2009SGR-345 and DGI MTM2011-27739-C04-02,
and by the Edinburgh Mathematical Society while visiting the University of Stirling.
J.~Z. Farkas was supported by a University of Stirling Research and Enterprise Support Grant and by the research project DGI MTM2011-27739-C04-02, while visiting the Universitat Aut\`{o}noma de Barcelona. 
We thank the referees for their valuable comments and suggestions.

\end{document}